\documentclass[10pt,leqno]{siamltex}

\usepackage{amsmath, amssymb, latexsym}
\usepackage{graphicx}
\usepackage{booktabs}
\usepackage{tikz}
\input epsf

\newtheorem{remark}{Remark}[section]

\newcommand{\p}{\partial}
\newcommand{\bd}{\partial}
\newcommand{\Grad}{\nabla}

\renewcommand{\div}{\operatorname{div}}

\newcommand{\bxi}{\boldsymbol{\xi}}

\newcommand{\R}{\mathbb{R}}
\newcommand{\norm}[1]{\|#1\|}
\usepackage{xcolor}

\definecolor{darkblue}{rgb}{0,0,0.6}

\newcommand{\rev}[1]{#1}
\def\bb{{\bf b}}

\def\bg{{\bf g}}

\def\bn{{\bf n}}
\def\bu{{\bf u}}
\def\bv{{\bf v}}

\def\bx{{\bf x}}

\def\bH{{\bf H}}

\def\bV{{\bf V}}

\def\hP{\hat{P}}

\def\hp{\hat{p}}

\def\hw{\hat{w}}
\def\hx{\hat{x}}
\def\hy{\hat{y}}
\def\hphi{\hat{\phi}}

\def\hbv{\hat{\bv}}
\def\hbx{\hat{\bx}}

\def\hQ{\hat{Q}}

\title{Raviart--Thomas Elements with Geometric Correction for Distorted Quadrilateral Meshes}

\author{
So--Hsiang Chou%
\thanks{
Department of Mathematics and Statistics,
Bowling Green State University,
Bowling Green, Ohio 43403-0221
({\tt chou@bgsu.edu}).
}
}

\begin{document}

\maketitle

\begin{center}
\small
Accepted for publication in the
\textit{Journal of Computational and Applied Mathematics}.
\end{center}
\begin{abstract}
Mixed finite element methods based on Raviart--Thomas spaces are
widely used for the numerical approximation of second--order
elliptic problems in flux form. On quadrilateral meshes, however,
the bilinear mapping from the reference element introduces a
spatially varying Jacobian, which may violate the inclusion
property $\mathrm{div}\,V_h \subset W_h$ for the standard
Raviart--Thomas spaces.

In this paper we propose a simple modification of the classical
Raviart--Thomas elements on quadrilateral meshes. The modification
consists of adding geometrically motivated correction terms to the
local basis functions in order to compensate for the geometric
distortion introduced by the bilinear mapping. The resulting spaces retain
the same dimension and degrees of freedom as the classical
Raviart--Thomas elements while restoring the compatibility property.

We present a general framework for constructing such modified
spaces and illustrate the approach by developing modified
versions of the lowest order and next--to--lowest order
Raviart--Thomas elements. Theoretical analysis establishes optimal
approximation properties under the standard shape--regularity
assumption for quadrilateral meshes. Numerical experiments on distorted meshes confirm the predicted convergence rates and show that the modified elements yield consistently improved accuracy over the classical Raviart--Thomas formulation as the geometric distortion increases.
\end{abstract}



\pagestyle{myheadings}
\thispagestyle{plain}
\markboth{So--Hsiang Chou}{Raviart--Thomas Elements with Geometric Correction for Distorted Quadrilateral Meshes}

%-------------------- The main part begins here

\section{Introduction}
\rev{
Mixed finite element methods based on the Raviart--Thomas spaces
introduced in the classical work of Raviart and Thomas
\cite{RaviartThomas1977}
are widely used for the numerical approximation of second--order
elliptic problems in flux form.
These methods provide locally conservative approximations and play an
important role in applications such as porous media flow and
subsurface transport.
On triangular meshes the classical Raviart--Thomas elements possess
well--known stability properties and optimal approximation estimates.

On quadrilateral meshes the situation is more subtle.
When the mapping from the reference square to the physical element is
bilinear rather than affine, the Jacobian of the mapping becomes
spatially varying.
As a consequence, the divergence of the mapped Raviart--Thomas space
no longer lies in the associated pressure space.
In particular, the inclusion property
\[
\div V_h \subset W_h
\]
may fail on distorted quadrilateral meshes.
This phenomenon has been widely discussed in the literature and
motivates the development of alternative mixed finite element
constructions for quadrilateral grids.
Classical mixed finite element theory may be found in
\cite{boffi2013mixed},
while the quadrilateral mapping pathology and related approximation
issues were analyzed extensively by Arnold, Boffi, and Falk
\cite{arnold2002approximation,arnold2005quadrilateral}.

The goal of this paper is to develop a simple framework for modifying
classical Raviart--Thomas spaces on quadrilateral meshes so that
optimal approximation properties are preserved under the standard
shape--regularity assumption.
The central idea is to modify the classical Raviart--Thomas basis
functions through carefully designed geometry-dependent correction
terms.
These corrections compensate for the distortion effects introduced by
the bilinear mapping while preserving the dimension and degrees of
freedom of the original space.
}%%% end of rev
Several approaches have been proposed to address these difficulties
for Raviart--Thomas spaces on quadrilateral meshes.
One direction analyzes the behavior of the classical Piola--mapped
spaces on distorted quadrilateral elements, identifying the loss of
compatibility between the velocity and pressure spaces.
Another direction constructs enriched or newly designed
$\bH(\div)$ finite element spaces; notable examples include the
quadrilateral elements of Arnold, Boffi, and Falk
\cite{arnold2002approximation,arnold2005quadrilateral}
and the minimal--dimension constructions of Arbogast and Correa
\cite{arbogast2016two}.
A third direction considers alternative mixed formulations on general
quadrilateral grids, such as the methods of Kwak, Pyo, and Hyon and
the earlier work of Shen
\cite{pyo2011mixed,HyonKwak2010,Shen1994}.

\rev{
Another important direction in finite element design is the use of
enrichment procedures, where additional local basis functions and
additional functionals are introduced in order to improve approximation
properties.
Such approaches have recently been developed for the
Crouzeix--Raviart element and related nonconforming methods; see, for
example,
\cite{Nudo2024JCAM,Nudo2024APNUM,DellAccioGuessabNudo2024}.
In contrast, the present construction does not enlarge the local
finite element space or introduce additional degrees of freedom.
Instead, the classical Raviart--Thomas basis functions are modified
through geometry-dependent correction terms designed to compensate for
the distortion induced by the bilinear mapping on quadrilateral
elements.
Thus the dimension and degrees of freedom of the classical
Raviart--Thomas space are preserved while restoring the compatibility
relation between the discrete velocity and pressure spaces.
}%%%%% end of rev

The remainder of the paper is organized as follows.
Section~2 introduces the geometric framework and reviews properties of
the Piola transformation on quadrilateral elements.
In Section~3 we construct the modified Raviart--Thomas spaces.
Section~4 develops the associated interpolation operators and
establishes their approximation properties.
Stability and error estimates for the mixed finite element method are
proved in Section~5.
Section~6 presents a post-processing procedure based on the proposed
modification.
Numerical experiments on distorted quadrilateral meshes are reported
in Section~7.
Concluding remarks are given in Section~8.
\section{Preliminaries}
In this section we review the geometric mapping for quadrilateral
elements and the classical Raviart--Thomas construction that will
be used in the development of the modified spaces.
Let ${\mathcal Q}_h=\{Q\}$ be a partition of a polygonal domain
$\Omega$ into convex quadrilaterals $Q$ with diameters not greater
than $h$. We take the unit square $\hQ=[0,1]^2$ in the
$\hx\hy$-plane as the reference element and label the four
vertices as $\hbx_i,i=1,2,3,4,$ in a counterclockwise order,
starting at the origin. Let $\hbx = (\hx,\hy)$ and $\bx = (x,y)$.
For a typical quadrilateral $Q\in {\mathcal Q}_h$ with vertices
$\bx_i,i=1,2,3,4$ arranged in a counterclockwise order, there
exists a unique bilinear bijection $F_Q$ from $\hQ$ onto $Q$
defined by
\begin{equation}
\bx = F_Q(\hbx) = \bx_1 + \bx_{21}\hx + \bx_{41}\hy + \bg\hx\hy,
\end{equation}
where $$ \bx_{ij} = \bx_i-\bx_j, \quad \bg = \bx_{12}+\bx_{34}. $$
Thus $ \bx_i = F_Q(\hbx_i), i = 1,2,3,4.$
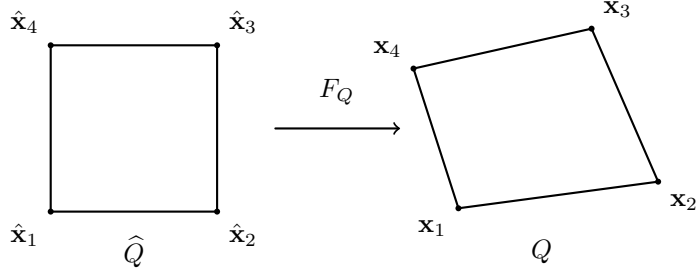
\begin{figure}[ht]
\centering
\begin{tikzpicture}[scale=2.2, line join=round, line cap=round]

% reference square
\draw[thick] (0,0) rectangle (1,1);
\fill (0,0) circle (0.018);
\fill (1,0) circle (0.018);
\fill (1,1) circle (0.018);
\fill (0,1) circle (0.018);

\node[below left=1pt]  at (0,0) {$\hat{\bx}_1$};
\node[below right=1pt] at (1,0) {$\hat{\bx}_2$};
\node[above right=1pt] at (1,1) {$\hat{\bx}_3$};
\node[above left=1pt]  at (0,1) {$\hat{\bx}_4$};
\node at (0.5,-0.24) {$\widehat Q$};

% mapping arrow
\draw[->,thick] (1.35,0.5) -- (2.10,0.5);
\node at (1.72,0.72) {$F_Q$};

% distorted quadrilateral (more gently tilted)
\coordinate (A) at (2.45,0.02);
\coordinate (B) at (3.65,0.18);
\coordinate (C) at (3.25,1.10);
\coordinate (D) at (2.18,0.86);

\draw[thick] (A)--(B)--(C)--(D)--cycle;

\fill (A) circle (0.018);
\fill (B) circle (0.018);
\fill (C) circle (0.018);
\fill (D) circle (0.018);

\node[below left=1pt]  at (A) {$\bx_1$};
\node[below right=1pt] at (B) {$\bx_2$};
\node[above right=1pt] at (C) {$\bx_3$};
\node[above left=1pt]  at (D) {$\bx_4$};
\node at (2.95,-0.24) {$Q$};

\end{tikzpicture}
\caption{The bilinear map $F_Q:\widehat Q\to Q$ from the reference square to a quadrilateral element.}
\label{fig:bilinear-map}
\end{figure}

Figure~\ref{fig:bilinear-map} illustrates the bilinear mapping from the
reference square to a general convex quadrilateral, which is the source
of the geometric distortion discussed below.
Because the Jacobian of the bilinear mapping is not constant,
the divergence of the Piola--transformed Raviart--Thomas space
does not necessarily lie in the discrete pressure space on
distorted quadrilateral meshes.

The Jacobian matrix
$DF_Q$ of $F_Q$ is given by
\begin{equation}\label{jacobian}
{DF}_Q = \begin{pmatrix}
             \frac{\p x}{\p\hx} & \frac{\p x}{\p\hy} \\
             \frac{\p y}{\p\hx} & \frac{\p y}{\p\hy}
             \end{pmatrix}
           = (\bx_{21}+\bg\hy, \bx_{41}+\bg\hx).
\end{equation}
In addition, the determinant $J_Q = \det{DF}_Q$ is a linear
function of $\hx$ and $\hy$:
\begin{equation}\label{convex}
J_Q(\hx,\hy) = \alpha+\beta\hx+\gamma\hy,
\end{equation}
where
\begin{equation*}
\alpha = \det(\bx_{21},\bx_{41}), \quad \beta =
\det(\bx_{21},\bg),
 \quad \gamma = \det(\bg,\bx_{41}).
\end{equation*}
The area of the quadrilateral $Q$ is equal to $J_Q(1/2,1/2)$,
since we have by the midpoint rule that
\begin{equation}\label{area}
 |Q| = \int_Q dxdy =
\int_0^1\int_0^1 J_Q(\hx,\hy)\,d\hx d\hy = J_Q(1/2,1/2).
\end{equation}
 Furthermore, $J>0$ by (\ref{convex}), since we have a convex quadrilateral.

Denote by $S_i$ the subtriangle of $Q$ with vertices
$\bx_{i-1},\bx_i$ and $\bx_{i+1}$ ( $\bx_0=\bx_4$). Let $h_Q$ be
the diameter of $Q$ and $\rho_Q= 2 \min_{1\le i\le 4}\{\mbox{
diameter of circle inscribed in } S_i\}$. Throughout the paper we
assume a {\it regular} family of partitions
 ${\mathcal Q}=\{ {\mathcal Q}_h \},$
i.e., there exists a positive constant $\sigma$, independent of
$h$, such that
\begin{equation}\label{regular}
 \frac{h_Q}{\rho_Q}\le \sigma, \,\quad \forall Q\in
{\mathcal Q}_h, \forall {\mathcal Q}_h\in {\mathcal Q}.
 \end{equation}

The relationships between different notions of shape regularity for
quadrilateral meshes were
clarified in \cite{ChouHe2002}.

The Piola transformation ${\mathcal P}_Q$ transforms a
vector-valued function on $\hQ$ to one on $Q$ by
\begin{equation}\label{Piola}
\bv = {\mathcal P}_Q\hbv = \frac1J{DF}\hbv\circ F^{-1},
\end{equation}
where we drop the subscript $Q$ for brevity. This transformation
preserves the $\bH(\div)$ space on the reference element. The
divergences are related by (cf. \cite{boffi2013mixed})
\begin{equation}\label{crux}
 \div\bv = \frac1J\div\hbv.
\end{equation}

%-------------------- 3rd Section begins here --------------------
 Let $\bV=\bH(\div;\Omega)$ and $W=L^2(\Omega)$. To facilitate the
 comparison between our approach and the standard one for constructing
 their approximation spaces, let us review the standard approach here.
  First one
defines reference spaces $\hat\bV=\bV(\hQ)$ and $\hat W=W(\hQ)$ on
$\hQ$. Typically one defines
 \[ \hat W:=\div \hat \bV.\]
 Then using the bilinear
transformation $F_Q$ to relate the scalar fields on $Q$ and $\hQ$,
we construct the local pressure space by $W(Q)=\{w:Q\to R:\quad
w\circ F_Q\in W(\hQ)\}$. In addition, using the Piola
transformation to relate vector fields on $Q$ and $\hQ$, we
construct the local velocity space by $\bV(Q)=\{\bv:Q\to R^2:\quad
\bv={\mathcal P}_Q\hat\bv,\quad\hat\bv\in \bV(\hQ)\}.$ Finally,
the global pressure space $W_h$ is defined as
\begin{equation}
W_h := \{w\in W :w|_Q\in W(Q),\quad\forall Q\in{\mathcal Q}_h \}
\end{equation}
and the global velocity space $\bV_h$ as
\begin{equation}
\bV_h=\{\bv\in \bV: \bv|_Q \in \bV(Q),\,\, \forall Q\in{\mathcal
Q}_h\}.
\end{equation}

For example, the Raviart--Thomas space $\bV_h\times W_h$ of order
$k, k\ge 0,$ for approximating functions in $\bV \times W$ is
constructed as follows. First one defines the velocity space $
\bV_k(\hQ)$ and pressure space $W_k(\hat Q)$ on $\hQ$ by
\begin{align}\label{def1} \bV_k(\hQ)&=RT_k(\hat Q)
:=Q_{k+1,k}\times Q_{k,k+1},\\ W_k(\hQ)&=\div \bV(\hat
Q)=Q_{k,k},
\end{align}
where $Q_{l,m}$ is the space of polynomials of degree at most $l$
and $m$ in $\hat x$ and $\hat y$, respectively. On a general $Q$,
the local Raviart--Thomas space is defined by
\begin{align}\label{def2}
\bV_k(Q)&=RT_k(Q):=\{\bv={\mathcal P}_Q\hat \bv:\qquad \hat\bv
\in\bV_k(\hQ)\},\\ W_k(Q)&:=\{w=\hw\circ F_Q^{-1}:\hw\in
W_k(\hQ)\}.
\end{align}
Finally the global Raviart--Thomas space is defined: for the
pressure space
\begin{equation}
W_h := \{w\in W :w|_Q\in W_k(Q),\quad\forall Q\in{\mathcal Q}_h
\},
\end{equation}
and for the velocity space $\bV_h$
\begin{equation}
\bV_h=RT_k:=\{\bv\in \bV: \bv|_Q = {\mathcal P}_Q\hbv,\,\hbv\in
\bV_k(\hQ),\,\, \forall Q\in{\mathcal Q}_h\},
\end{equation}
\begin{remark}
Due to (\ref{crux}), $\div \bV_h$ is not contained in $W_h$. This
is the property we aim to restore when modifying $\bV_h$, a
central theme of the paper.
\end{remark}
For further properties of these spaces, see \cite{boffi2013mixed,Shen1994}

The Raviart--Thomas projection ${\mathcal
RT}^{(k)}_h:\bH^1(\Omega)\to\bV_h,\, k\ge 0$ is defined as
follows. (We drop the order index if there is no confusion.) Given
a function $\hat\bu$ on $\hQ$ a unique interpolant $\hat {\mathcal
RT}\hat \bu$ is defined by the following degrees of freedom:
\begin{align}\label{interp}
\int_{\hat l_i} (\hat\bu-\hat {\mathcal RT}\hat \bu)\cdot
\hat\bn\phi ds=0,\qquad \forall \phi\in P_k(\hat
l_i),\,i=1,2,3,4,\\ \label{interp1} \int_{\hQ}(\hat\bu-\hat
{\mathcal RT}\hat\bu)\cdot\hat \bv\,d\hat x\hat y=0\qquad
\forall \hat\bv\in Q_{k-1,k}\times Q_{k,k-1}\quad (k\ge 1),
\end{align}
where $\hat l_i$ are the four edges of $\hQ$ and $P_k(\hat l_i)$
is the space of polynomials in $s$ of degree at most $k$. On $Q$,
${\mathcal RT}_Q \bu={\mathcal P}_Q\hat {\mathcal RT}\hat
\bu.$ A projection operator $P_h:W\to W_h$ can be defined
\cite{boffi2013mixed} in the same manner. First define the local $L^2$
projection $\hP$ on $\hQ$ by
\begin{align}
\int_{\hQ} (\hP\hp-\hp)\hat w\,d\hx d\hy=0 \qquad \forall \hat
w\in W(\hQ)
\end{align}
and then set
\begin{gather}
P_Q\phi = (\hP\hphi)\circ F_Q^{-1}, \qquad \forall \phi\in L^2(Q),
\end{gather}
where $\hphi = \phi\circ F_Q$. Finally, we define
\begin{equation}
{\mathcal RT}_h\bv|_Q = {\mathcal RT}_Q\bv, \qquad P_h\phi|_Q
= P_Q\phi.
\end{equation}

 \subsection{Essential Properties Leading to
Optimal Order Approximations} Let us collect some essential
properties that would lead to optimal order approximation in the
$L_2$ norm  when using the $RT_k$ space as an approximation
space. This is a necessary step to motivate and clarify the later
construction of the modified $RT_k$ space.

\begin{lemma}\label{e1}
The following orthogonality relations hold
\begin{gather}
(\div(\bu-{\mathcal RT}_h^{(k)}\bu),w) = 0, \qquad
\forall\bu\in\bV,
\quad \forall w\in W_h, \label{orthogonality1} \\
(\div\bv,\phi-P_h\phi) = 0, \qquad \forall\bv\in\bV_h, \quad
\forall\phi\in W. \label{orthogonality2}
\end{gather}
\end{lemma}

From now on, the letter $C$ will denote a generic positive
constant which is independent of $h$. It may have different values
in different places, especially when used in proof.
\begin{lemma}\label{e2}
The following estimates are valid for regular partitions:
\begin{align} \norm{P_h\phi}_0 &\le
C\norm{\phi}_0, \qquad \forall\phi\in W, \label{bound_P}\\
\norm{\phi-P_h\phi}_0 &\le Ch^{k+1}|\phi|_{k+1},\quad  k\ge 0,
\qquad \forall \phi\in H^{k+1}(\Omega). \label{approx_P1}
\end{align}
\end{lemma}
\begin{proof} Inequality
(\ref{approx_P1}) can be found in \cite[pp. 106--108]{girault1986finite}.
\end{proof}

 Lemmas \ref{e1}--\ref{e2} are sufficient to show the following optimal
 order approximation in the $L^2$ norm.
 \begin{theorem}\label{e3}
If the partitions are regular then the optimal--order approximation
holds for the $RT_k$ interpolation operator:
\begin{gather}\label{Falk}
||\bu-{\mathcal RT}^{(k)}_h\bu||_0 \le Ch^{k+1}\,|\bu|_{k+1}.
\end{gather}
\end{theorem}

Inequality (\ref{Falk}) is valid \cite{boffi2013mixed} for
regular triangular and parallelogram partitions, but invalid for quadrilateral partitions \cite{arnold2002approximation, arnold2005quadrilateral}. So we now know the $L^2$ interpolation
error is optimal for the Raviart--Thomas operators, but what is
surprising is that they also showed that unless the partitions are
almost parallelogram, error in the divergence has only suboptimal
order of convergence
\begin{equation*}
 \norm{\div(\bu-{\mathcal RT}_h^{(k)} \bu)}_0 \le
Ch^k\,|\div\bu|_{k+1}.
\end{equation*} In particular, there is no convergence
for $k=0$. They then proposed a family of new elements $ABF_k$
(Arnold--Boffi--Falk space of order $k$), which has optimal order
of convergence. For $k=0$, the new space has two more degrees of
freedom than  $RT_k$ does, and for $k=1$ four more degrees of
freedom have to be added to the set (\ref{interp}). We show in
this paper that alternative spaces can be constructed (alternative in the sense that they preserve the same dimensions
as $RT_k$ while maintaining optimal-order convergence). {\it The crux of the
matter is to have a new look at the reference element and at the
relation $\div\bv=\frac{1}{J}\div\hat \bv$.} As a consequence,
their framework does not cover our elements and their necessary
and sufficient condition for a given space to have optimal order
of convergence does not apply to our case. Under a fixed reference
element, this necessary and sufficient condition states that
$RT_k$ space must be enlarged so that they contain the so--called $R_{k+1}$ space (see \cite{arnold2005quadrilateral}) for its definition).

Lemmas \ref{e1}--\ref{e2} are not sufficient to show optimal order
approximation of divergence. Our main objective is to construct a
modified space $\bV_h=MRT_k$ for which the discrete divergence maps naturally
into the pressure space $W_h$, while at the same time not violating
the validity of Lemmas \ref{e1}--\ref{e2}. In other words, Lemmas
\ref{e1}--\ref{e2} together with this compatibility relation imply the optimal order approximation of
the divergence term. The two approaches, ours and that of \cite{arnold2005quadrilateral}, actually complement each
other and make the optimal order theory complete under the
shape--regularity condition. Our approach is more natural from a
finite volume application viewpoint \cite{ChouTang2000} since we insist on
$\div \bV_h\subset W_h$ (cf. the last paragraph in Section
\ref{sub}), while the traditional finite element approach
maintains $\div \hat\bV=\hat W$ ($\div\bV_h\subset W_h$ is
violated instead).

Let us now summarize the main theme of this paper: Given a
$RT_k$ velocity space and its accompanying pressure space $W_k$,
we construct a modified space $MRT_k$ with the same
pressure space $W_k$ such that $\mbox{dim}\, RT_k=\mbox{dim}\,
MRT_k$ and $\div MRT_k\subset W_k$.
 \subsection{Construction of the Modified $RT_0$ Space}\label{sub}

Modified Raviart–-Thomas space $MRT_0$ has appeared in the literature, beginning with the technical report of Shen \cite{Shen1994}, and was later developed and analyzed in detail by Kwak and collaborators \cite{pyo2011mixed,HyonKwak2010}. In this work, $MRT_0$ is used as a building block, and we briefly recall its construction for completeness. Given a velocity field $\hat\bu=(a+b\hx,c+d\hy)$
in $RT_0(\hQ)$, we create a new $\hat\bu_M$ in the following
form:
\begin{eqnarray*}
a+b\hx+H_1\hx(\hx-1),\\ c+d\hy+H_2\hy(\hy-1),
\end{eqnarray*}
where $H_1$ and $H_2$ are determined by insisting on $Q$ that $
\div\bu_M$ is in $W_k$ ($k=0$). Recalling that $\bu_M ={\mathcal
P}_Q\hat \bu_M$, we thus require
\[ \div \bu_M=\frac 1 J\div\hat\bu_M \in  W_k\]
or equivalently
\[ (\alpha+\beta\hx+\gamma \hy)t_0=(b+d)+\left((H_1(2\hx-1)+H_2(2\hy-1)\right)\]
for some constant $t_0$. Comparing the corresponding coefficients,
we can determine $t_0,H_1$ and $H_2$ in terms of $a,b,c,d$ and
$\alpha,\beta,\gamma$. It is easily verified upon using
(\ref{area}) that
\[t_0=\frac{b+d}{|Q|},\qquad  H_1=\frac{\beta}{2}t_0,
\qquad H_2=\frac{\gamma}2t_0.\] Hence $MRT_0(\hQ)$ is
easily determined as the four dimensional vector space
\begin{equation}\label{def}
 MRT_0(\hQ):=\{\hat\bu:\hat
u_1=(a+bx+\frac{(b+d)\beta}{2|Q|}\hx(\hx-1),
c+dy+\frac{(b+d)\gamma}{2|Q|}\hy(\hy-1)\}.
\end{equation}
 Note
that strictly speaking, $MRT_0(\hQ)$ should be denoted as
$MRT_0(\hQ,Q)$. This is slightly different from the
traditional reference element approach. Nevertheless we can still
create global space $\bV_h=MRT_0$ from the local ones in
the usual way. Once we have set up the fact that $\div
\bV_h\subset W_h$ or $\div \bV_h=W_h$ in this case, the optimal
estimate comes naturally. For instance, if we approximate $\div
\bu=f$ on $Q$ by $\div\bu_h=\bar f$, where $\bar
f=\frac{1}{|Q|}\int_Q fdxdy$ and $\bu_h\in MRT_0$. Then it
is immediate that $||\div \bu-\div\bu_h||_0\le Ch||\div \bu||_1$.
This already hints at why the above principle can lead to optimal
$\div$ approximation. {\it That is, the property $\div
\bV_h\subset W_h$ plus an approximation property of a local $L^2$
projection should be sufficient for optimal order approximation of
the divergence.} One may wonder why the bubble functions were
added. The reason is that we want to maintain the property
(\ref{orthogonality1}) by not having any new effect through
(\ref{interp}). (See Lemma \ref{Lemma3.50} below for the general
case.) In the remaining sections we will develop these ideas
further for $RT_1$, which is much more involved. The $RT_0$ case follows by a simpler argument.

The approach taken in the present paper is somewhat different.
Instead of enlarging the polynomial space or introducing new
degrees of freedom, we modify the local Raviart--Thomas basis
functions by adding correction terms that depend on the
geometric distortion of the quadrilateral element.
These corrections restore optimal $L^2$ approximation properties
while preserving the dimension of the classical Raviart--Thomas
space.

The resulting modified space satisfies the discrete compatibility
relation
\[
\div V_h \subset W_h,
\]
which is central to the stability of the mixed formulation.
From this perspective the present construction may be viewed as
a geometric correction of the standard Raviart--Thomas space
on distorted quadrilateral meshes.

%-------------------- 5th Section begins here --------------------
\section{Construction of the Modified $RT_1$ Spaces}
It is  known \cite{arnold2005quadrilateral} that quadrilateral $RT_k,k\ge 0$
spaces do not have optimal rate of convergence in the $\bH(\div)$
norm, unless one imposes an almost parallelogram condition. Such a
condition is too restrictive and our goal here is to create a
better velocity space $\bV_h=MRT_1$ that attains optimal
order convergence rate. This space should not differ too much
from $RT_1$. Given each element in $RT_1$ we shall construct
an accompanying element by adding four bubble like functions so
that $\div{\bf V}_h$ is contained in $W_h=\{ w:\Omega\to R: \hat
w\in Q_{1,1}\}$. Notice that we do not want to change the pressure
space and it is still the same $W_h$ corresponding to the $RT$
space, and furthermore the dimensions of $MRT_1$ and
$RT_1$ are the same. We mention that once the construction
process is illustrated for $RT_1$, the case of $RT_0$ is the
same and actually much far easier.

 Now we want to choose the correction bubble functions so that
$\div{\bf V}_h\subset W_h$. To decide what to add, let's look at
the following table
\[\begin{array}
{llll} 1&\hx&\hx^2&\hx^3\\ \hy&\hx\hy&\hx^2\hy&\hx^3\hy\\
\hy^2&\hx\hy^2&\hx^2\hy^2&\hx^3\hy^2\\
\hy^3&\hx\hy^3&\hx^2\hy^3&\hx^3\hy^3
\end{array}
\]
and take a cue on how $MRT_0$ was obtained. Recall in
Section one the modified $RT_0$ was obtained from $RT_0$ by
adding $\hx^2$ and $\hy^2$ terms, and the desire of keeping the
degrees of freedom intact led to the choice of $\hx(\hx-1)$ and
$\hy(\hy-1)$. In doing so, we used all the terms on the left of
the diagonal line connecting $\hx^2$ and $\hy^2$ except for the
$\hx\hy$ term. So for modifying $RT_1$ we would conjecture to
add those terms on the left of the diagonal line connecting
$\hx^3\hy$ and $\hx\hy^3$ except for the $\hx^2\hy^2$ term. Thus
we add
\[ {\hat  x}^3{\hat  y},{\hat  x}^2{\hat  y},{\hat  x}^3,{\hat  y}^3{\hat  x},
{\hat  y}^2{\hat  x},{\hat  y}^3\] and the desire of not changing
the degrees of freedom leads us to using bubble functions that
vanish on the boundary. It is remarkable that this conjecture
turns out to be true and we now elaborate. Let $\hat \bu=(\hat
u,\hat v)\in RT_1(\hat Q)$ where
\begin{align}\label{original}
 \hat u&=a+b{\hat x}+c{\hat x}^2+d{\hat y}+e{\hat
x}{\hat y}+f{\hat x}^2{\hat y},\\ \label{original2}\hat
v&=A+B{\hat y}+C{\hat y}^2+D{\hat x}+E{\hat x}{\hat y}+F{\hat
x}{\hat y}^2.
\end{align}
We define a modified velocity $\hat \bu_M=(\hat u_M,\hat v_M)$
where
\begin{align}
\label{new}\hat u_M =a+b{\hat x}+c{\hat x}^2+d{\hat y}+e{\hat
x}{\hat y}+f{\hat x}^2{\hat y}+H_1{\hat x}^2({\hat x}-1){\hat
y}+G_1{\hat x}^2({\hat x}-1),\\
\label{new2}\hat v_M=A+B{\hat y}+C{\hat y}^2+D{\hat x}+E{\hat
x}{\hat y}+F{\hat x}{\hat y}^2+H_2{\hat y}^2({\hat y}-1){\hat
x}+G_2{\hat y}^2({\hat y}-1).
\end{align}
Hence
\begin{align*}
\div\hat\bu_M&= b+2c{\hat x}+e{\hat y}+2f{\hat x}{\hat
y}+H_1(3{\hat x}^2-2{\hat x}){\hat y}+G_1(3{\hat x}^2-2{\hat x})\\
&+B+2C{\hat y}+E{\hat x}+2F{\hat x}{\hat y}+H_2(3{\hat
y}^2-2{\hat y}){\hat x}+G_2(3{\hat y}^2-2{\hat y}).
\end{align*}
Let us now impose the condition $\frac 1J \div\hat
\bu_M=\div\bu_M\in W_1(Q)$, the local pressure space (recall
$\bu_M={\mathcal P}_Q\hat\bu_M)$, so that on $\hat Q$
\[ \div\hat \bu_M=J \times \mbox{ a bilinear function},\]
i.e.,
\begin{equation}\label{primary} \div{\hat \bu_M}=(\alpha+\beta
 {\hat x}+\gamma
 {\hat y})(t_1+t_2{\hat x}+t_3{\hat y}+t_4{\hat x}{\hat y})
\end{equation}
for some bilinear function $t_1+t_2{\hat x}+t_3{\hat y}+t_4{\hat
x}{\hat y}.$ The central idea is to find
$t_i,i=1,\ldots,4,H_i,i=1,2,$ and $G_i,i=1,2$ in terms of
$b,c,e,f,B,C,E,F$ and $\alpha,\beta,\gamma$. Now the right hand
side of (\ref{primary}) is
\begin{eqnarray*}
RHS&=& \alpha t_1+(\alpha t_2+\beta t_1){\hat x}+(\alpha
t_3+\gamma t_1){\hat y}+(\alpha t_4+\beta t_3+\gamma t_2){\hat
x}{\hat y}\\ &&+\beta t_2{\hat x}^2+\beta t_4{\hat x}^2{\hat y}+
\gamma t_3{\hat y}^2+\gamma t_4{\hat x}{\hat y}^2,
\end{eqnarray*}
while the left hand side is
\begin{eqnarray*}
LHS&=& (b+B)+(2c+E-2G_1){\hat x}+(2C+e-2G_2){\hat
y}\\&&+(2f+2F-2H_1-2H_2){\hat x}{\hat y}\\ &&+3G_1{\hat x}^2+
3H_1{\hat x}^2{\hat y}+ 3G_2{\hat y}^2+3H_2{\hat x}{\hat y}^2.
\end{eqnarray*}
Equating leads to
\begin{align}\label{sett1}
 \alpha t_1&=b+B, \\\label{sett2}
\alpha t_2+\beta t_1 &=2c+E-2G_1,
\\\label{sett3}
\alpha t_3+\gamma t_1&=e+2C-2G_2,\\\label{sett4} \alpha t_4+\beta
t_3+\gamma t_2&=2f+2F-2H_1-2H_2,
\\\label{sett5}
\beta t_2&=3G_1,\\\label{sett6}\beta t_4 &=3H_1,\\\label{sett7}
\gamma t_3&= 3G_2,\\ \label{sett8}\gamma t_4&=3H_2.
\end{align}
We can express the $H_i$ and $G_i$ values in the first four
equations by the $t$ values in the last four equations. Thus
\begin{align}\label{t1}
t_1&=\frac{b+B}{\alpha},\\\label{t2}
t_2&=\frac{2c+E-\beta(b+B)/\alpha}{\alpha+\frac{2}{3}\beta},\\\label{t3}
t_3&=\frac{e+2C-\gamma(b+B)/\alpha}{\alpha+\frac{2}{3}\gamma},\\\label{t4} t_4&=\frac{2f+2F-\beta t_3-\gamma
t_2}{\alpha+\frac{2\beta}{3}+\frac{2\gamma}{3}},\\ \label{H1}
 H_1&=\frac 13\beta
t_4,\\\label{H2} H_2&=\frac 13 \gamma t_4,\\ \label{G1}
G_1&=\frac 13\beta t_2,\\ \label{G2} G_2&= \frac{1}{3}\gamma t_3.
\end{align}
\rev{
We note that the quantities
\[
\alpha,\qquad
\alpha+\frac23\beta,\qquad
\alpha+\frac23\gamma,\qquad
\alpha+\frac23\beta+\frac23\gamma
\]
appearing in \eqref{t1}--\eqref{t4} have a natural geometric interpretation.
Since
\[
J_Q(\hat x,\hat y)=\alpha+\beta\hat x+\gamma\hat y,
\]
we have
\[
\alpha
=
J_Q(0,0),
\qquad
\alpha+\frac23\beta
=
J_Q\!\left(\frac23,0\right),
\]
\[
\alpha+\frac23\gamma
=
J_Q\!\left(0,\frac23\right),
\qquad
\alpha+\frac23\beta+\frac23\gamma
=
J_Q\!\left(\frac23,\frac23\right).
\]
Since $J_Q(\hat x,\hat y)>0$ on $\hat Q$ for convex quadrilateral
elements satisfying the shape--regularity assumption, these quantities
are uniformly positive and bounded away from zero. Therefore the
correction coefficients in \eqref{t1}--\eqref{t4} are well-defined
and uniformly controlled.
}
In summary, for each $Q$ we have from $ RT_1(\hQ)$ a modified
local space  $MRT_1(\hQ,Q):=\{\hat\bu_M:\,
\hat\bu_M:\hQ\to R \mbox{ as defined above}, \mbox{ given }
\hat\bu\in RT_1(\hQ)\}$. The space $MRT_1(\hQ, Q)$
depends on $Q$ through the geometric quantities $\alpha,\beta$ and
$\gamma$. Nevertheless, we can still construct as usual the space
$MRT_1(Q):=\{\bu_M:Q\to R^2:\bu_M={\cal P}_Q\hat\bu_M,
\hat \bu_M\in MRT_1(\hQ,Q)\}$ Finally, we can define
\[  MRT_1=\{\bu\in \bH(\div;\Omega):
\bu|_Q\in MRT_1(Q)\quad\forall Q\in{\mathcal Q}_h\}\] and
we see that the property
\[ \div MRT_1\subset W_h\]
holds.
\subsection{Intersection of $MRT_k(\hQ)$ and
$RT_k(\hQ)$}For simplicity, let us denote $MRT_k(\hQ,
Q)$ as $MRT_k(\hQ)$, when no confusion can arise. We next
show that for $k=0,1$ the subspace of divergence free functions in
$RT_k(\hQ)$ is contained in $MRT_k(\hQ)$ and vice versa.
\begin{lemma}\label{Lemma3.0} For $k=0,1$ the following two sets are equal:
\[\{\hat\bu\in RT_k(\hQ): \div \hat\bu=0\}=\{\hat\bu_M\in
MRT_k(\hQ): \div \hat\bu_M=0\}.\] Furthermore, for these
two sets $\hat\bu=\hat\bu_M$.
\end{lemma}
\begin{proof}
One can get the conclusion quite easily by using (\ref{def}) for
the case of $k=0$. For $k=1$, let $\hat\bu_M$ be in $MRT_1(\hQ)$ and $\div\hat\bu_M=0$. Then by \eqref{primary}, all
$t_i$ are zero. Hence by \eqref{sett1}--\eqref{sett8}, we have all
$H_i$ and $G_i$ values equal to zero and
\begin{align}\label{temp} b+B&=0, \qquad 2c+E=0,\\\label{temp2}
e+2C&=0,\qquad 2f+2F=0.
\end{align}
Hence $\hat\bu_M$ is also in $RT_1(\hQ)$ and
$\hat\bu_M=\hat\bu$. Conversely, if $\hat\bu\in RT_1(\hQ)$ and
$\div\hat\bu=0$, then we have \eqref{temp}--\eqref{temp2}. Using
\eqref{sett1}--\eqref{sett8}, we can see that all $H_i$ and $G_i$
values are zero. Hence $\hat\bu=\hat\bu_M$ and $\div\hat\bu_M=0$
as well. This completes the proof.

\end{proof}

\begin{remark}\label{Remark3.0} It is now easy to see the curl
 of a biquadratic polynomial lies in
the intersection of $MRT_1(\hat Q)$ and $RT_1(\hat Q)$.
This result will be used later when we invoke a Bramble-Hilbert
type lemma.
\end{remark}
The above lemma can be used to see that some of the basis
functions of $RT_0(\hQ)$ can be retained for ${\bf
M}RT_0(\hQ)$.
 To generate a basis function for $MRT_1$ by modifying
an existing one, it is helpful to have the following lemma in
mind.
\begin{lemma}
Let $\hat\bu$ be in $MRT_1(\hQ)\cap RT_1(\hQ)$. Then
$\int_{\hQ}\div\hat\bu d\hx d\hy=0$ if and only if $\div\hat\bu=0$
on $\hQ$.
\end{lemma}
\begin{proof}
We show the `only if' part. Using \eqref{original}--\eqref{new2},
\eqref{t1}--\eqref{G2}, and the fact that all the $H_i$ and $G_i$
values are zero, we have
\begin{align}\label{tem1}2c+E&=\frac\beta\alpha(b+B)\\\label{tem2}
e+2C&=\frac\gamma\alpha(b+B)\\ \label{tem3}2f+2F&=0.
\end{align}
%Changed 12/30/03
Applying the above relations we have
\[\int_{\hQ}\div\hat\bu_M d\hx d\hy=\frac{|Q|(b+B)}{\alpha}=0\]
and hence $b+B=0$. This together with \eqref{tem1}--\eqref{tem3}
imply $\div\hat\bv=0$.
\end{proof}

%%%%%%%%%%%%%%%%%%%%%%%%%%%%%%%%%
\section{Interpolation and Approximation Properties}

In this section we study the approximation properties of the
modified Raviart--Thomas spaces introduced in the previous
section. We first define the associated interpolation operator
and establish the corresponding degrees of freedom. These
properties are then used to derive approximation estimates and
the convergence result for the mixed finite element method.

Given a vector field ${\bf {\hat u}}$ on the reference element
$\hat Q$ we shall show that for $k=0$ or $k=1$ there is a unique
$\hat \Pi {\bf \hat u}\in MRT_k({\hat Q})$ such that for
any edge $\hat e$ of $\hQ$
\begin{align}
\label{dof1}
\int_{\hat e}({\bf \hat u}-\hat \Pi {\bf \hat u})\cdot\hat {\bf
n}wd\sigma=0,\qquad \mbox{ for all }w\in P_i(\hat e),\, i=0,1,
\\
\label{dof2}
\int_{\hat Q}({\bf \hat u}-\hat \Pi {\bf \hat
u})\cdot\hat {\bf v}dxdy=0,\qquad \mbox{ for all }\hat {\bf v}\in
Q_{k-1,k}\times Q_{k,k-1} \ (k=1).
\end{align}

\rev{

\begin{lemma}\label{Lemma3.DOF}
Equations (\ref{dof1})--(\ref{dof2}) define a unisolvent set
of degrees of freedom for the space $MRT_k(\hat Q),k=0,1$.
\end{lemma}

\begin{proof}
For $k=0$ the conclusion is immediate, since $MRT_0(\hat Q)$
has the same basis structure and the same number of degrees of freedom
as the classical $RT_0$ element.

We now consider the case $k=1$.
The number of functionals in \eqref{dof1}--\eqref{dof2} coincides with the
dimension of $MRT_1(\hat Q)$, which is eight. Hence it suffices to show that
if all degrees of freedom vanish, then the corresponding function is identically
zero.

Let
\[
{\bf \hat u}_M=(\hat u_M,\hat v_M)\in MRT_1(\hat Q)
\]
be given by
\begin{align}
\hat u_M
&=
a+b\hat x+c\hat x^2+d\hat y
+e\hat x\hat y
+f\hat x^2\hat y
+H_1\hat x^2(\hat x-1)\hat y
+G_1\hat x^2(\hat x-1),
\\
\hat v_M
&=
A+B\hat y+C\hat y^2+D\hat x
+E\hat x\hat y
+F\hat x\hat y^2
+H_2\hat y^2(\hat y-1)\hat x
+G_2\hat y^2(\hat y-1).
\end{align}

Assume that all degrees of freedom in \eqref{dof1}--\eqref{dof2} vanish.
We derive the resulting homogeneous linear system for the coefficients.

On the edge $\hat y=0$, the normal vector is $(0,-1)^T$.
Hence
\[
\hat v_M(\hat x,0)=A+D\hat x.
\]
Using \eqref{dof1} with $w=1$ and $w=\hat x$, we obtain
\[
\int_0^1 (A+D\hat x)\,d\hat x=0,
\qquad
\int_0^1 (A+D\hat x)\hat x\,d\hat x=0.
\]
Thus
\[
A+\frac12D=0,
\qquad
\frac12A+\frac13D=0.
\]
Solving gives
\[
A=D=0.
\]

Similarly, on the edge $\hat x=0$ we have
\[
\hat u_M(0,\hat y)=a+d\hat y.
\]
Again using \eqref{dof1} with $w=1$ and $w=\hat y$, we obtain
\[
a=d=0.
\]

Next consider the edge $\hat y=1$.
Since
\[
\hat v_M(\hat x,1)
=
A+B+C+(D+E+F)\hat x,
\]
the vanishing moments imply
\[
A+B+C+\frac12(D+E+F)=0,
\]
\[
\frac12(A+B+C)+\frac13(D+E+F)=0.
\]
Using $A=D=0$, we obtain
\[
B+C+\frac12(E+F)=0,
\]
\[
\frac12(B+C)+\frac13(E+F)=0.
\]
Hence
\[
B+C=0,
\qquad
E+F=0.
\]

Finally, on the edge $\hat x=1$ we have
\[
\hat u_M(1,\hat y)
=
a+b+c+(d+e+f)\hat y.
\]
Using $a=d=0$ and imposing the vanishing moments gives
\[
b+c+\frac12(e+f)=0,
\]
\[
\frac12(b+c)+\frac13(e+f)=0.
\]
Thus
\[
b+c=0,
\qquad
e+f=0.
\]

We next use the interior moments \eqref{dof2}.
Choosing successively the basis functions
\[
(1,0)^T,\quad
(\hat x,0)^T,\quad
(0,1)^T,\quad
(0,\hat y)^T,
\]
we obtain four additional equations.

First,
\[
\int_{\hat Q}\hat u_M\,d\hat x\,d\hat y=0.
\]
Substituting the expressions above and using
\[
a=d=0,\qquad b+c=0,\qquad e+f=0,
\]
gives
\[
\frac16 b+\frac1{12}e-\frac1{12}G_1-\frac1{24}H_1=0.
\]

Second,
\[
\int_{\hat Q}\hat x\,\hat u_M\,d\hat x\,d\hat y=0,
\]
which yields
\[
\frac1{12}b+\frac1{24}e-\frac1{15}G_1-\frac1{30}H_1=0.
\]

Third,
\[
\int_{\hat Q}\hat v_M\,d\hat x\,d\hat y=0.
\]
Using
\[
A=D=0,\qquad B+C=0,\qquad E+F=0,
\]
we obtain
\[
\frac16 B+\frac1{12}E-\frac1{12}G_2-\frac1{24}H_2=0.
\]

Fourth,
\[
\int_{\hat Q}\hat y\,\hat v_M\,d\hat x\,d\hat y=0,
\]
which gives
\[
\frac1{12}B+\frac1{24}E-\frac1{15}G_2-\frac1{30}H_2=0.
\]

Using the defining relations
\begin{align}
G_1 &= \frac13\beta t_2,
&
G_2 &= \frac13\gamma t_3,
\\
H_1 &= \frac13\beta t_4,
&
H_2 &= \frac13\gamma t_4,
\end{align}
together with
\begin{align}
t_2 &=
\frac{2c+E-\beta(b+B)/\alpha}
{\alpha+\frac23\beta},
\\
t_3 &=
\frac{e+2C-\gamma(b+B)/\alpha}
{\alpha+\frac23\gamma},
\\
t_4 &=
\frac{2f+2F-\beta t_3-\gamma t_2}
{\alpha+\frac23\beta+\frac23\gamma},
\end{align}
and recalling
\[
c=-b,\qquad
C=-B,\qquad
f=-e,\qquad
F=-E,
\]
we obtain a homogeneous linear system
\[
M
\begin{pmatrix}
b\\ e\\ B\\ E
\end{pmatrix}
=0,
\]
where
\[
M=
\begin{pmatrix}
\frac16+\frac{\beta^2}{18\alpha_1}
&
\frac1{12}+\frac{\beta\gamma}{72\alpha_3}
&
-\frac{\beta^2}{18\alpha\alpha_1}
&
-\frac{\beta}{36\alpha_3}
\\[2mm]
\frac1{12}+\frac{2\beta^2}{45\alpha_1}
&
\frac1{24}+\frac{\beta\gamma}{45\alpha_3}
&
-\frac{2\beta^2}{45\alpha\alpha_1}
&
-\frac{\beta}{45\alpha_3}
\\[2mm]
-\frac{\beta\gamma}{36\alpha_3}
&
-\frac{\gamma^2}{18\alpha\alpha_2}
&
\frac16+\frac{\gamma^2}{18\alpha_2}
&
\frac1{12}+\frac{\beta\gamma}{72\alpha_3}
\\[2mm]
-\frac{\beta\gamma}{45\alpha_3}
&
-\frac{2\gamma^2}{45\alpha\alpha_2}
&
\frac1{12}+\frac{2\gamma^2}{45\alpha_2}
&
\frac1{24}+\frac{\beta\gamma}{45\alpha_3}
\end{pmatrix},
\]
with
\[
\alpha_1=\alpha+\frac23\beta,\qquad
\alpha_2=\alpha+\frac23\gamma,\qquad
\alpha_3=\alpha+\frac23\beta+\frac23\gamma.
\]

A direct symbolic computation yields
\[
\det(M)
=
\frac{
\left(
9\alpha^2+6\alpha\beta+6\alpha\gamma+4\beta\gamma
\right)^2
}{
466560\,\alpha^2\alpha_1^2\alpha_2^2\alpha_3^2
}.
\]

Since
\[
J_Q(\widehat x,\widehat y)
=
\alpha+\beta\widehat x+\gamma\widehat y
>0
\quad\text{on }\widehat Q,
\]
the quantities
\[
\alpha,\qquad
\alpha_1,\qquad
\alpha_2,\qquad
\alpha_3
\]
are uniformly positive under the assumed shape--regularity condition.
Hence
\[
\det(M)\neq0.
\]

The homogeneous system therefore admits only the trivial solution,
\[
b=e=B=E=0.
\]
It follows that
\[
c=f=C=F=0,
\]
and consequently
\[
G_1=G_2=H_1=H_2=0.
\]

Thus all coefficients vanish and
\[
{\bf \hat u}_M\equiv{\bf 0}.
\]
Hence the degrees of freedom
\eqref{dof1}--\eqref{dof2}
are unisolvent for $MRT_1(\hat Q)$.
\end{proof}
} %%%% end of rev
Using these degrees of freedom we define the local interpolation
operator $\Pi_Q$ on each element $Q$ and the global interpolation
operator $\Pi_h$ in the usual manner.
\begin{lemma}
Let $\Pi_h$ be the interpolation operator and $P_h$ the $L^2$
projection onto $W_h$. Then
\[
\div(\Pi_h \bu) = P_h(\div \bu).
\]
\end{lemma}

\begin{proof}
The result follows from the definition of the degrees of freedom
and the construction of the modified basis functions. For any $w_h \in W_h$, using the definition of the degrees of freedom and integration by parts on each element $Q$, we have
\[
\int_Q \operatorname{div}(\Pi_h u)\, w_h
= \int_Q \operatorname{div} u \, w_h.
\]
Since $\operatorname{div}(\Pi_h u), P_h(\operatorname{div} u) \in W_h$, this implies
\[
\operatorname{div}(\Pi_h u) = P_h(\operatorname{div} u).
\]
\end{proof}

\section{Stability and Error Analysis}
In this section we establish the boundedness of the interpolation operator $\Pi_h$
and the stability of the mixed finite element pair
$MRT_k \times W_h$.

Let us first recall that for shape regular partitions we have the
following upper bounds \cite[p. 105]{girault1986finite}:
\begin{equation}\label{bound_jacobian}
\norm{{DF}_Q}_{\infty,\hQ}\le Ch_Q, \qquad
\norm{{DF}_Q^{-1}}_{\infty,Q} \le Ch_Q^{-1}
\end{equation}
where $\|M\|_{\infty,K}:=\sup_{\bx\in K}\|M(x)\|$, the supremum of
the spectral norm of the matrix function $M$. In addition, the
function $J_Q$ satisfies \cite{girault1986finite}:
\begin{equation}\label{bound_det}
\|J_Q\|_{\infty,\widehat Q} \le Ch_Q^2, \qquad
\inf_{(\hat x,\hat y)\in\widehat Q} J_Q(\hat x,\hat y) \ge Ch_Q^2.
\end{equation}
Now we show the following lemma relating seminorms on $Q$ and $\hat Q$.

\begin{lemma}\label{lemma3.1}
Let ${\bf u}=(u,v)^T$ and $\hat{\bf u}=(\hat u,\hat v)^T$ be related by
${\bf u}={\mathcal P}_Q\hat{\bf u}$. For a regular family of quadrilateral
partitions, there exists a constant $C$, independent of $h$, such that
\begin{align}
\label{inq1}
\Big\|\frac{\partial\hat u}{\partial \hat y}\Big\|_{0,\hat Q}
+\Big\|\frac{\partial\hat v}{\partial \hat x}\Big\|_{0,\hat Q}
&\le
Ch_Q\,|{\bf u}|_{1,Q},
\qquad \forall {\bf u}\in {\bf H}^1(Q),\\
\label{inq2}
\Big\|\frac{\partial^2\hat u}{\partial \hat y^2}\Big\|_{0,\hat Q}
+\Big\|\frac{\partial^2\hat v}{\partial \hat x^2}\Big\|_{0,\hat Q}
&\le
Ch_Q^2\,|{\bf u}|_{2,Q},
\qquad \forall {\bf u}\in {\bf H}^2(Q).
\end{align}
\end{lemma}
\begin{proof}
To simplify the notation, we use subscripts for partial derivatives;
for example, $u_x=\frac{\partial u}{\partial x}$.
By the definition of the Piola transformation, we have
\begin{align}
\label{good}
\hat u &= y_{\hat y}u-x_{\hat y}v,\\
\label{bad}
\hat v &= -\,y_{\hat x}u+x_{\hat x}v.
\end{align}
Hence
\[
\hat u_{\hat y}
=
x_{\hat y}y_{\hat y}u_x
+
(y_{\hat y})^2u_y
-
(x_{\hat y})^2v_x
+
x_{\hat y}y_{\hat y}v_y,
\]
since the term $y_{\hat y\hat y}u-x_{\hat y\hat y}v$ vanishes.
Noting that
$x_{\hat x},x_{\hat y},y_{\hat x},y_{\hat y},x_{\hat x\hat y},y_{\hat x\hat y}$
are all of order $h$, and that $J\ge Ch^2$, we obtain
\[
\|\hat u_{\hat y}\|_{0,\hat Q}\le Ch\,|{\bf u}|_{1,Q}.
\]
Similarly,
\[
\|\hat v_{\hat x}\|_{0,\hat Q}\le Ch\,|{\bf u}|_{1,Q}.
\]
This proves \eqref{inq1}.

For \eqref{inq2}, differentiating \eqref{good} and \eqref{bad} once more,
we see that $\hat u_{\hat y\hat y}$ and $\hat v_{\hat x\hat x}$ are linear
combinations of $u_{xx},u_{xy},u_{yy},v_{xx},v_{xy},v_{yy}$, where the
coefficients involve products of first derivatives of the mapping such as
$x_{\hat y}y_{\hat y}$ and are therefore of order $h_Q^2$. This yields
\eqref{inq2}.
\end{proof}

The following lemma is needed to prove the boundedness of $\Pi_h$.

\begin{lemma}\label{Lemma3.4} Let $\Pi_h$ be the interpolation
operator from $\bH^1(\Omega)$ to $MRT_k, k=0$ or $1$. Then
\[||{\bf u}-\Pi_h {\bf u}||_0\le Ch|{\bf u}|_1\qquad  \forall {\bf
u}\in {\bf H}^1(\Omega).\]
\end{lemma}
\begin{proof}
To prove this theorem, we borrow an idea from  \cite[p. 255]{girault1986finite},
which uses a decomposition to reduce vector consideration to
scalar one so that a Bramble-Hilbert type lemma can be used.

Let ${\bf u}\in {\bf H}^1(Q)$ and let $\hat{\bf u}$ be its Piola pullback on $\hat Q$. Then we can use the
decomposition $\hat{ \bf u}=\nabla q+ \mbox{curl} \psi$, where $q$
can be determined by the elliptic problem
\begin{eqnarray*}
 \Delta q&=&\div {\bf\hat  u}\qquad
\mbox{ on } \hat Q\\ q&=&0\qquad \qquad \mbox{ on } \partial \hat
Q.
\end{eqnarray*}
Because $\hat Q$ is convex the estimate $||q||_{2,\hat Q}\le
C||\div \hat {\bf  u}||_{0,\hat Q}$ holds. Using the above
decomposition $\mbox{curl} \psi=\hat {\bf  u}-\nabla q$ and
differentiating the first component with respect to $\hat y$ and
second one with respect to $\hat x$, we derive

\[
||\frac{\partial^2\psi} {\partial \hat y^2}||_{0,\hat
Q}+||\frac{\partial^2\psi} {\partial \hat x^2}||_{0,\hat Q}\le C
\left( ||\frac{\partial\hat u_1}{\partial \hat y}||_{0,\hat
Q}+||\frac{\partial\hat u_2}{\partial \hat x}||_{0,\hat Q}+||\div
\hat {\bf u}||_{0,\hat Q}\right).\] Since
\[
||{\bf\hat  u}-\hat {\Pi}  \hat {\bf  u}||_{0,\hat Q}\le ||\nabla
q-\hat{\Pi} \nabla q||_{0,\hat Q}+||\mbox{curl}\psi-\hat {\Pi}
\mbox{curl} \psi||_{0,\hat Q}\] we estimate the right hand side
terms separately. Using a Bramble-Hilbert type Lemma \cite[p.
106]{girault1986finite} for quadrilaterals with the corresponding seminorms
$[f]^2_i:=||\frac{\partial^i f}{\partial \hat x^i}||_{0,\hat
Q}^2+||\frac{\partial^i f}{\partial \hat y^i}||_{0,\hat Q}^2$, we
see that
\[||\nabla
q-\hat{\Pi} \nabla q||_{0,\hat Q}\le C [q]_1,\] since the quantity
$\nabla q-\hat{\Pi} \nabla q$ is zero for constant $q$. On the
other hand,
\[||\mbox{curl}\psi-\hat {\Pi}
\mbox{curl} \psi||_{0,\hat Q}\le C[\psi]_2,\] since the quantity
$\mbox{curl}\psi-\hat {\Pi} \mbox{curl} \psi$ is zero for bilinear
$\psi$. This is true for both the Raviart--Thomas and modified
Raviart--Thomas interpolation operators since the curl of a
bilinear function is of the form $(a+bx, A+By)$ with $b+B=0$ (cf.
Lemma \ref{Lemma3.0}). Hence
\[||\hat {\bf  u}-\hat {\Pi}  \hat{\bf  u}||_{0,\hat Q}\le
 C \left( ||\frac{\partial\hat u_1}{\partial \hat y}||_{0,\hat
Q}+||\frac{\partial\hat u_2}{\partial \hat x}||_{0,\hat Q}+||\div
\hat{\bf  u}||_{0,\hat Q}\right).\]

Now, by the norm-equivalence induced by the Piola transformation and the
shape regularity of the partition, we have
\begin{eqnarray*}
 \|{\bf u}-\Pi {\bf u}\|_{0,Q}
 &\le&
 C\|\hat{\bf u}-\hat\Pi \hat{\bf u}\|_{0,\hat Q}\\
 &\le&
 C\left(
 \Big\|\frac{\partial \hat u_1}{\partial \hat y}\Big\|_{0,\hat Q}
 +
 \Big\|\frac{\partial \hat u_2}{\partial \hat x}\Big\|_{0,\hat Q}
 +
 \|\div \hat{\bf u}\|_{0,\hat Q}
 \right).
\end{eqnarray*}
Using \eqref{inq1}, we obtain
\[
\Big\|\frac{\partial \hat u_1}{\partial \hat y}\Big\|_{0,\hat Q}
+
\Big\|\frac{\partial \hat u_2}{\partial \hat x}\Big\|_{0,\hat Q}
\le Ch\,|{\bf u}|_{1,Q}.
\]
Moreover, since
\[
\div {\bf u}=\frac1J \div \hat{\bf u},
\]
we have
\begin{eqnarray*}
\|\div \hat{\bf u}\|_{0,\hat Q}^2
&=
\int_{\hat Q} |\div \hat{\bf u}|^2\,d\hat x d\hat y
=
\int_{\hat Q} J^2\,|(\div {\bf u})\circ F_Q|^2\,d\hat x d\hat y\\
&=
\int_Q J\,|\div {\bf u}|^2\,dxdy
\le Ch^2\|\div {\bf u}\|_{0,Q}^2.
\end{eqnarray*}
Hence
\[
\|\div \hat{\bf u}\|_{0,\hat Q}\le Ch\,\|\div {\bf u}\|_{0,Q}.
\]
Therefore
\[
\|{\bf u}-\Pi {\bf u}\|_{0,Q}
\le
Ch\bigl(|{\bf u}|_{1,Q}+\|\div {\bf u}\|_{0,Q}\bigr)
\le
Ch\,|{\bf u}|_{1,Q},
\]
which proves the result after summing over all $Q$.
 \end{proof}
 \begin{lemma}\label{Lemma3.50}
 Let $\Pi_h$ be the interpolation
operator from $\bH^1(\Omega)$ to $MRT_k, k=0$ or $1$. Then
 \begin{align}\label{div11}
 (\div(\bu-\Pi_h\bu),w)&=0\qquad \forall \bu\in \bV, \forall w\in W_h,\\
 \label{div22}
 ||\div(\bu-\Pi_h\bu)||_0&\le Ch^{k+1}|\div\bu|_{k+1},
 \end{align}
 where $k=0$ for the space $MRT_0$ and $k=1$ for the space
 $MRT_1$.
 \end{lemma}
 \begin{proof}
 Observe that
 \begin{eqnarray*}
 (\div(\bu-\Pi_h\bu),w)_Q&=&(\div(\hat\bu-\hat\Pi\hat\bu),\hw)_{\hQ}\\
 &=&-\int_{\hQ}(\hat\bu-\hat\Pi\hat\bu)\cdot \nabla \hw\,d\hx
 d\hy+\sum_{i=1,4}\int_{\hat l_i}(\hat\bu-\hat\Pi\hat\bu)\cdot{\bf
 n}ds\\
 &=&0,
 \end{eqnarray*}
 where we have used (\ref{dof1})--(\ref{dof2}).
 Summing over $\hQ$ proves (\ref{div11}).

 Using the fact $\div \bV_h\subset W_h$, the approximation
 property of $P_h$ (\ref{approx_P1}), and (\ref{div11}), we have
 \begin{eqnarray*}
 ||\div(\bu-\Pi_h\bu)||_0^2&=&
\left(\div(\bu-\Pi_h\bu),\div(\bu-\Pi_h\bu)\right)\\
 &=&(\div(\bu-\Pi_h\bu),\div\bu)\\
 &=&\left(\div(\bu-\Pi_h\bu),\div\bu-P_h\div\bu\right)\\
 &\le&||\div(\bu-\Pi_h\bu)||_0||P_h\div\bu-\div \bu||_0\\
 &\le&Ch^{k+1}||\div(\bu-\Pi_h\bu)||_0|\div\bu|_{k+1}.
 \end{eqnarray*}
 This completes the proof of (\ref{div22}).
 \end{proof}

\begin{lemma}\label{Lemma3.5} The operator $\Pi_h$ is bounded, i.e.,
\[||\Pi_h {\bf v}||_{{\bf H}(\div)}\le C ||{\bf v}||_1
\qquad \forall {\bf v}\in {\bV}.\]
\end{lemma}
\rev{
\begin{proof}
Using Lemma \ref{Lemma3.4}
\begin{eqnarray*}
\|\Pi_h {\bf v}||_0&\le&||\Pi_h {\bf v}-{\bf v}||_0+||{\bf
v}||_0\\ &\le& C||{\bf v}||_1.
\end{eqnarray*}
Furthermore,  by \eqref{div11},
\[
(\div({\bf v}-\Pi_h {\bf v}),q_h)=0
\qquad\forall q_h\in W_h.
\]
Taking $q_h=\div\Pi_h {\bf v}\in W_h$ gives
\[
\|\div \Pi_h {\bf v}\|_0^2
=
(\div {\bf v},\div\Pi_h {\bf v})
\le
\|\div {\bf v}\|_0\,\|\div\Pi_h {\bf v}\|_0.
\]
Hence
\[
\|\div\Pi_h {\bf v}\|_0
\le
\|\div {\bf v}\|_0.
\]
Combining the above estimates yields
\[
\|\Pi_h {\bf v}\|_{H(\div)}
\le
C\|{\bf v}\|_1.
\]
\end{proof}
}%%%end of rev

\begin{lemma}\label{Lemma3.6} There exists a positive constant
$\beta_0$ independent of $h$ such that for all $w_h\in W_h$
\[ ||w_h||_0\le \beta_0 \sup_{{\bf v}_h\in {\bf
V}_h}\frac{(w_h,\div {\bf v}_h)}{{||\bf v}_h||_{{\bf H}(\div)}}.\]
\end{lemma}
\rev{
\begin{proof}
Given $w_h\in W_h$, by the continuous inf-sup condition for
$H(\div;\Omega)\times L^2(\Omega)$, there exists
${\bf w}\in {\bf H}^1(\Omega)$ such that
\[
\div {\bf w}=w_h,
\qquad
\|{\bf w}\|_1\le C\|w_h\|_0.
\]
Using the commuting property \eqref{div11},
\[
(\div \Pi_h{\bf w},w_h)
=
(\div {\bf w},w_h)
=
\|w_h\|_0^2.
\]
Therefore
\[
\sup_{{\bf v}_h\in V_h}
\frac{(w_h,\div {\bf v}_h)}
{\|{\bf v}_h\|_{H(\div)}}
\ge
\frac{\|w_h\|_0^2}
{\|\Pi_h{\bf w}\|_{H(\div)}}.
\]
Applying Lemma \ref{Lemma3.5} gives
\[
\|\Pi_h{\bf w}\|_{H(\div)}
\le
C\|{\bf w}\|_1
\le
C\|w_h\|_0,
\]
which proves the discrete inf-sup condition.
\end{proof}
} %%%%% end of rev
Hence $MRT_k\times W_h$ is a stable mixed finite element pair.

We next record a lemma that controls the difference between the classical
$RT_1$ interpolant and the modified interpolant.

%%%%%%%%%%%%%%%%%%%%%%%%%%%%%%%%%%%%%%%%%%%%%%%%%%%%
%%%%%%%%%%%%%%%%%%%%%%%%%%%%%%%%%%%%%%%%%%%%%%%%%%%%%
\begin{lemma}\label{Lemma5.6}
Let $\hat{\bf v}\in RT_1(\hat Q)$ and let $\hat\Pi \hat{\bf v}\in MRT_1(\hat Q)$.
Assume that
\[
\div \hat v_M = J(\hat x,\hat y)\,\hat q(\hat x,\hat y),
\qquad \hat q \in Q_{1,1}(\hat Q),
\]
where $\hat v_M=\hat\Pi \hat v$.
There exists a positive constant $C$ independent of $h$ such that
\[
\|\hat{\bf v}-\hat\Pi \hat{\bf v}\|_{0,\hat Q}
\le
C(|\beta|+|\gamma|)\,|\hat q|_{1,\hat Q}.
\]
\end{lemma}

\begin{proof}
Let $\hat{\bf v}=(\hat v_1,\hat v_2)^T \in RT_1(\hat Q)$ and
$\hat\Pi \hat{\bf v}\in MRT_1(\hat Q)$.
From the unisolvence relations (cf. Lemma \ref{Lemma3.DOF}), the difference
$\hat{\bf v}-\hat\Pi \hat{\bf v}$ depends only on the coefficients
$H_1,H_2,G_1,G_2$ associated with the non-affine correction terms.

More precisely, one obtains
\[
\hat{\bf v}-\hat\Pi \hat{\bf v}
=
\begin{pmatrix}
\frac{G_1}{2}\hat x - \frac{G_1}{2}\hat x^2
+ \frac{H_1}{2}\hat x\hat y - \frac{H_1}{2}\hat x^2\hat y
- H_1 \hat x^2(\hat x-1)\hat y - G_1 \hat x^2(\hat x-1) \\
\frac{G_2}{2}\hat y - \frac{G_2}{2}\hat y^2
+ \frac{H_2}{2}\hat x\hat y - \frac{H_2}{2}\hat x\hat y^2
- H_2 \hat y^2(\hat y-1)\hat x - G_2 \hat y^2(\hat y-1)
\end{pmatrix}.
\]

By the construction of the modified basis functions,
\[
H_1=\tfrac13 \beta t_4,\quad G_1=\tfrac13 \beta t_2,
\qquad
H_2=\tfrac13 \gamma t_4,\quad G_2=\tfrac13 \gamma t_3,
\]
where $\hat q=t_1+t_2\hat x+t_3\hat y+t_4\hat x\hat y \in Q_{1,1}(\hat Q)$.

Substituting these relations into the above expression and using the
boundedness of the polynomial basis on $\hat Q$, we obtain
\[
\|\hat{\bf v}-\hat\Pi \hat{\bf v}\|_{0,\hat Q}
\le
C\bigl(
|\beta|(|t_2|+|t_4|)
+
|\gamma|(|t_3|+|t_4|)
\bigr).
\]

Since all norms are equivalent on the finite-dimensional space
$Q_{1,1}(\hat Q)$, we have
\[
|t_2|+|t_3|+|t_4|
\le
C\,|\hat q|_{1,\hat Q}.
\]

Therefore,
\[
\|\hat{\bf v}-\hat\Pi \hat{\bf v}\|_{0,\hat Q}
\le
C(|\beta|+|\gamma|)\,|\hat q|_{1,\hat Q},
\]
which completes the proof.
\end{proof}

%%%%%%%%%%%%%%%%%%%%
Now we can finally prove the following interpolation error estimate.

%\begin{theorem}\label{thm:MRT1-L2}
\begin{theorem}\label{thm5.7}
Let $\Pi_h$ be the interpolation operator from ${\bf H}^2(\Omega)$
to $MRT_1$. Then
\[
\|{\bf u}-\Pi_h{\bf u}\|_{0}
\le
Ch^2\bigl(|{\bf u}|_{2}+|\div{\bf u}|_{1}\bigr),
\qquad
{\bf u}\in {\bf H}^2(\Omega),\ \div{\bf u}\in H^1(\Omega).
\]
\end{theorem}
\rev{
\begin{proof}
Let $R_h$ denote the classical Raviart--Thomas interpolation operator
onto $RT_1$.  We emphasize that $R_h$ is used here only as an auxiliary
comparison operator for the $L^2$ approximation of the vector field.
The loss of optimality on distorted quadrilateral meshes concerns the
divergence approximation of the classical $RT_1$ interpolant, not the
$L^2$ approximation estimate used below.

On each element $Q$, write
\[
{\bf u}-\Pi_h{\bf u}
=
({\bf u}-R_h{\bf u})+(R_h{\bf u}-\Pi_hR_h{\bf u})+\Pi_h(R_h{\bf u}-{\bf u}).
\]
Hence, using the local boundedness of $\Pi_h$ in $L^2$,
\[
\|{\bf u}-\Pi_h{\bf u}\|_{0,Q}
\le
C\|{\bf u}-R_h{\bf u}\|_{0,Q}
+
\|R_h{\bf u}-\Pi_hR_h{\bf u}\|_{0,Q}.
\]
The first term is estimated by the standard $L^2$ approximation property
of the classical Raviart--Thomas interpolant on shape--regular
quadrilateral meshes:
\[
\|{\bf u}-R_h{\bf u}\|_{0,Q}
\le
C h_Q^2 |{\bf u}|_{2,Q}.
\]
This estimate is distinct from the generally suboptimal estimate for
$\|\operatorname{div}({\bf u}-R_h{\bf u})\|_{0,Q}$ on distorted quadrilaterals.

It remains to estimate the correction term.
Let
\[
R_h{\bf u}|_Q=P_Q\hat v,
\qquad
\hat v\in RT_1(\hat Q).
\]
Then, by the scaling associated with the Piola transformation,
\[
\|R_h{\bf u}-\Pi_hR_h{\bf u}\|_{0,Q}
\le
C\|\hat v-\hat\Pi\hat v\|_{0,\hat Q}.
\]
By Lemma~5.6,
\[
\|\hat v-\hat\Pi\hat v\|_{0,\hat Q}
\le
C(|\beta|+|\gamma|)|\hat q|_{1,\hat Q},
\]
where
\[
\hat q=(\operatorname{div} R_h{\bf u})|_Q\circ F_Q .
\]
Since the mesh is shape regular,
\[
|\beta|+|\gamma|\le C h_Q^2.
\]
Moreover,
\[
\operatorname{div} R_h{\bf u}=P_h(\operatorname{div}{\bf u}),
\]
and the local $H^1$ stability of the projection on the finite-dimensional
space $W_h(Q)$ gives
\[
|\hat q|_{1,\hat Q}
\le
C |P_h(\operatorname{div}{\bf u})|_{1,Q}
\le
C |\operatorname{div}{\bf u}|_{1,Q}.
\]
Therefore,
\[
\|R_h{\bf u}-\Pi_hR_h{\bf u}\|_{0,Q}
\le
C h_Q^2 |\operatorname{div}{\bf u}|_{1,Q}.
\]
Combining the two estimates gives
\[
\|{\bf u}-\Pi_h{\bf u}\|_{0,Q}
\le
C h_Q^2
\bigl(
|{\bf u}|_{2,Q}
+
|\operatorname{div}{\bf u}|_{1,Q}
\bigr).
\]
Summing over all elements $Q$ yields
\[
\|{\bf u}-\Pi_h{\bf u}\|_0
\le
C h^2
\bigl(
|{\bf u}|_2+|\operatorname{div}{\bf u}|_1
\bigr).
\]
This completes the proof.
\end{proof}
} %%%%%%% end of rev

We now apply the above estimates to the mixed method.
Consider the following second-order elliptic boundary value
problem on a bounded polygonal domain $\Omega$ in $\R^2$ with the
boundary $\bd\Omega$:
\begin{equation}\label{problem}
\left\{ \begin{aligned} -\div K(\bx)\Grad p &= f  \qquad \mbox{ in
} \Omega, \\
                p&= 0  \qquad \mbox{ on } \bd\Omega.
\end{aligned} \right.
\end{equation}
Here $\bn$ is the outward unit normal vector to $\bd\Omega$, and
the coefficient $K$ is a symmetric and uniformly positive-definite
matrix, i.e., there exist two positive constants $\alpha_1$ and
$\alpha_2$ such that
\begin{equation}\label{coer}
\alpha_1\bxi^T\bxi \le \bxi^TK(\bx)\bxi \le \alpha_2\bxi^T\bxi,
\qquad \forall\bxi\in\R^2,\ \bx\in\overline{\Omega}.
\end{equation}
For brevity we will often omit dependency of the coefficients
$K,\bb$, and $c$ on the space variable $\bx$.

Rewrite the problem (\ref{problem}) as a system of first-order
partial differential equations
\begin{subequations}\label{mixed_system}
\begin{align}
K^{-1}\bu + \Grad p&= 0, \qquad \mbox{ in } \Omega, \\ \div\bu&=
f, \qquad \mbox{ in } \Omega, \\ p &= 0, \qquad \mbox{ on }
\bd\Omega.
\end{align}
\end{subequations}

Now let us introduce the function spaces
\begin{gather}
\bV=\bH({\rm div};\Omega) = \{\bv\in({\bf L}^2(\Omega)): \div\bv\in
L^2(\Omega)\}, \\ W = L^2(\Omega).
\end{gather}
Then the associated weak formulation for (\ref{mixed_system}) is
to find $(\bu,p)\in\bV\times W$ such that
\begin{subequations}\label{weak_form}
\begin{gather}
(K^{-1}\bu,\bv) - (\div\bv,p)= 0, \qquad \forall\bv\in\bV, \\
(\div\bu,w) = (f,w) , \qquad \forall w\in W,
\end{gather}
\end{subequations}
where $(\cdot,\cdot)$ denotes the standard inner product in
 $L^2(\Omega)$ or $(L^2(\Omega))^2$.

 Now let $\bV_h\subset \bV$ and $W_h\subset W$
 and consider the mixed method: Find $(\bu_h,p_h)\in \bV_h\times
 W_h$ such that

 \begin{subequations}\label{mixed}
\begin{gather}
(K^{-1}\bu_h,\bv_h) - (\div\bv_h,p_h)= 0, \qquad
\forall\bv_h\in\bV_h,
\\ (\div\bu_h,w_h) = (f,w_h) , \qquad \forall w_h\in W_h,
\end{gather}
\end{subequations}

\begin{theorem}\label{Theorem5.8}
Let ${\bf u}$ and $p$ be the solution of \eqref{weak_form} and let $({\bf u}_h,p_h)$ be the
solution of \eqref{mixed} with $V_h=MRT_k\subset V$ and
$W_h=\{\,w:\hat w\in Q_{k,k}\,\}\subset W$, where $k=0$ or $1$.
Assume that the partitions are regular. Then
\begin{align}
\|{\bf u}-{\bf u}_h\|_0
&\le
\begin{cases}
Ch\|{\bf u}\|_1, & k=0,\\[1mm]
Ch^2\bigl(|{\bf u}|_2+|\div {\bf u}|_1\bigr), & k=1,
\end{cases}
\label{eq:5.16new}\\
\|\div({\bf u}-{\bf u}_h)\|_0
&\le
Ch^{k+1}|\div {\bf u}|_{k+1},
\label{eq:5.17new}\\
\|p-p_h\|_0
&\le
Ch^{k+1}\bigl(\|p\|_1+\|{\bf u}\|_{k+1}\bigr).
\label{eq:5.18new}
\end{align}
\end{theorem}

\begin{proof}
Subtracting \eqref{weak_form} from \eqref{mixed}, we have
\[
(K^{-1}({\bf u}-{\bf u}_h),{\bf v}_h) -(p-p_h,\div {\bf v}_h)=0
\qquad \forall {\bf v}_h\in V_h,
\]
and
\[
(\div({\bf u}-{\bf u}_h),q_h)=0
\qquad \forall q_h\in W_h.
\]

Hence
\begin{align*}
C\|{\bf u}-{\bf u}_h\|_0^2
&\le (K^{-1}({\bf u}-{\bf u}_h),{\bf u}-{\bf u}_h) \\
&= (K^{-1}({\bf u}-{\bf u}_h),\Pi_h{\bf u}-{\bf u}_h)
   +(K^{-1}({\bf u}-{\bf u}_h),{\bf u}-\Pi_h{\bf u}) \\
&= (p-p_h,\div(\Pi_h{\bf u}-{\bf u}_h))
   +(K^{-1}({\bf u}-{\bf u}_h),{\bf u}-\Pi_h{\bf u}) \\
&= (K^{-1}({\bf u}-{\bf u}_h),{\bf u}-\Pi_h{\bf u}) \\
&\le C\|{\bf u}-{\bf u}_h\|_0\,\|{\bf u}-\Pi_h{\bf u}\|_0 .
\end{align*}
Therefore,
\[
\|{\bf u}-{\bf u}_h\|_0 \le C\|{\bf u}-\Pi_h{\bf u}\|_0.
\]

For $k=0$, Lemma~5.2 gives
\[
\|{\bf u}-{\bf u}_h\|_0 \le C\|{\bf u}-\Pi_h{\bf u}\|_0 \le Ch\|{\bf u}\|_1.
\]

For $k=1$, Theorem~5.7 gives
\[
\|{\bf u}-{\bf u}_h\|_0 \le C\|{\bf u}-\Pi_h{\bf u}\|_0
\le Ch^2\bigl(|{\bf u}|_2+|\div {\bf u}|_1\bigr).
\]

This proves \eqref{eq:5.16new}. The proof of \eqref{eq:5.17new} is the same as
before:
\[
\|\div({\bf u}-{\bf u}_h)\|_0^2
=
(\div({\bf u}-{\bf u}_h),\div {\bf u}-P_h\div {\bf u})
\le
\|\div({\bf u}-{\bf u}_h)\|_0\,\|P_h\div {\bf u}-\div {\bf u}\|_0,
\]
and therefore
\[
\|\div({\bf u}-{\bf u}_h)\|_0 \le Ch^{k+1}|\div {\bf u}|_{k+1}.
\]
\rev{
The pressure estimate \eqref{eq:5.18new} follows from the discrete inf-sup condition.
Indeed,
\[
\|p-p_h\|_0
\le
C\sup_{{\bf v}_h\in {\bf v}_h}
\frac{(p-p_h,\div {\bf v}_h)}
{\|{\bf v}_h\|_{H(\div)}}.
\]
Using the error equation,
\[
(p-p_h,\div {\bf v}_h)
=
(K^{-1}({\bf u}-{\bf u}_h),{\bf v}_h),
\]
hence
\[
|(p-p_h,\div {\bf v}_h)|
\le
C\|{\bf u}-{\bf u}_h\|_0\|{\bf v}_h\|_{H(\div)}.
\]
Therefore
\[
\|p-p_h\|_0
\le
C\|{\bf u}-{\bf u}_h\|_0.
\]
Combining this with \eqref{eq:5.16new} proves \eqref{eq:5.18new}.
}%%%%%%%%%%%% end of rev
\end{proof}

For $k=1$, the estimate involves the additional term $|\div {\bf u}|_1$
arising from the geometry-dependent modification.

\section{Post-processing of classical $RT_0$ solutions}
In addition to constructing modified Raviart--Thomas spaces,
the modification framework can also be used as a post-processing
tool for solutions obtained with classical mixed finite element
methods. In particular, solutions computed with the standard
$RT_0$ method can be locally corrected using the same geometric
modification procedure described earlier. This leads to an
improved approximation without increasing the number of global
degrees of freedom.
In this section we address the following question: Let $\bV_h$ be
any of the $RT_k, k=0,1$ and  let $\bu_h$ in
$\bV_h$ be a solution of \begin{equation}\label{dvieq1}
(\div\bu_h,q_h)=(f,q_h) \qquad \forall q_h\in W_h,
\end{equation}where $W_h$ is the pressure space corresponding to
$\bV_h$ as defined before. Now let $\bu$ be a solution of
\begin{equation}\label{diveq2}
 (\div\bu,q)=(f,q)\qquad \forall q\in W. \end{equation} Since we
know $\div(\bu-\bu_h)$ does not have optimal order, we ask if we
can use the modified element ${\mathcal M}\bu_h$ as a
post-processor for $\bu_h$. Unfortunately, the answer is yes only
in the $RT_0$ case. Nevertheless, the optimal order of
approximation can be achieved by this element.
\begin{theorem}\label{postprocessing}
Let $\bu_h\in \bV_h=RT_0$ be a solution of (\ref{dvieq1}) and
let $\tilde\bu_h$ be the modified element ${\mathcal M}\bu_h$ as
defined before. Then
\begin{equation}\label{post}
(\div\tilde \bu_h,q_h)=(f,q_h)\qquad \forall q_h\in W_h
\end{equation}
and \begin{equation}\label{repeat} ||\div(\bu-\tilde\bu_h)||_0\le
Ch|\div \bu\|_{1}, \end{equation} where $\bu$ is a solution of
(\ref{diveq2}).
\end{theorem}
\begin{proof}First recall on $\hQ$ that $\hat {\mathcal M}\hat \bu_h
 =\hat\bu_h+\hat\bb_h$ where
$\hat\bb_h$ is the bubble function part in (\ref{def}). Now we
need only observe by direct calculation that
\[\int_{\hQ}\div\hat\bb_h\, q_h d\hx d\hy=0,\]
and hence
\begin{eqnarray*}
(\div \bu_h,q_h)&=&\sum_Q\int_Q\div\bu_h\, q_h dxdy\\
&=&\sum_Q\int_{\hQ}\div\hat\bu_h\hat q_h d\hx d\hy\\
&=&\sum_Q\int_{\hQ}\div\hat {\mathcal M}\hat\bu_h\,q_h d\hx
d\hy-\int_{\hQ}\div\hat\bb_h\, q_h d\hx d\hy\\
&=&\sum_Q\int_{\hQ}\div\hat {\mathcal M}\hat\bu_h\,q_h d\hx
d\hy=(\div\tilde \bu_h,q_h)
\end{eqnarray*}
after transforming back.

Since \[(\div({\bf u}-\tilde {\bf u}_h),q_h)=0\quad  \forall q_h\in W_h,\] the estimate (\ref{repeat}) follows
by the same argument used in the proof of \eqref{eq:5.17new}, together with
the approximation property of the $L^2$ projection. This completes the proof.
\end{proof}

\section{Numerical Experiments}
\rev{
In this section we present several numerical experiments illustrating
the behavior of the modified Raviart--Thomas element $MRT_1$
on distorted quadrilateral meshes.
The experiments are designed to investigate three related aspects:

\begin{itemize}
\item the compatibility property
\[
\div V_h \subset W_h,
\]
\item the approximation behavior of the interpolation operator,
\item and the performance of the resulting mixed finite element method
for elliptic problems.
\end{itemize}

Throughout the experiments we consider a family of distorted
quadrilateral meshes generated from a uniform partition of $\Omega$
into $N\times N$ squares, followed by an alternating
compression--stretch perturbation controlled by a distortion parameter
$\alpha\in[0,1)$.
For each element indexed by $(i,j)$, let
\[
s = (-1)^{i+j}.
\]
The interior vertices are then perturbed so that adjacent elements are
distorted in opposite directions, producing a non-affine quadrilateral
mesh while preserving mesh regularity.
The case $\alpha=0$ corresponds to an affine mesh, whereas larger
values of $\alpha$ produce increasingly distorted quadrilateral
elements.

Figure~\ref{fig:distortedmesh} illustrates a representative distorted
mesh used in the computations.
\begin{figure}[h!]
\centering
\begin{tikzpicture}[scale=5]

% Bottom row
\coordinate (P00) at (0.00,0.00);
\coordinate (P10) at (0.25,0.00);
\coordinate (P20) at (0.50,0.00);
\coordinate (P30) at (0.75,0.00);
\coordinate (P40) at (1.00,0.00);

% Row 1
\coordinate (P01) at (0.00,0.25);
\coordinate (P11) at (0.18,0.25);
\coordinate (P21) at (0.57,0.25);
\coordinate (P31) at (0.68,0.25);
\coordinate (P41) at (1.00,0.25);

% Row 2
\coordinate (P02) at (0.00,0.50);
\coordinate (P12) at (0.32,0.50);
\coordinate (P22) at (0.43,0.50);
\coordinate (P32) at (0.82,0.50);
\coordinate (P42) at (1.00,0.50);

% Row 3
\coordinate (P03) at (0.00,0.75);
\coordinate (P13) at (0.18,0.75);
\coordinate (P23) at (0.57,0.75);
\coordinate (P33) at (0.68,0.75);
\coordinate (P43) at (1.00,0.75);

% Top row
\coordinate (P04) at (0.00,1.00);
\coordinate (P14) at (0.25,1.00);
\coordinate (P24) at (0.50,1.00);
\coordinate (P34) at (0.75,1.00);
\coordinate (P44) at (1.00,1.00);

% Horizontal edges
\draw (P00)--(P10)--(P20)--(P30)--(P40);
\draw (P01)--(P11)--(P21)--(P31)--(P41);
\draw (P02)--(P12)--(P22)--(P32)--(P42);
\draw (P03)--(P13)--(P23)--(P33)--(P43);
\draw (P04)--(P14)--(P24)--(P34)--(P44);

% Vertical edges
\draw (P00)--(P01)--(P02)--(P03)--(P04);
\draw (P10)--(P11)--(P12)--(P13)--(P14);
\draw (P20)--(P21)--(P22)--(P23)--(P24);
\draw (P30)--(P31)--(P32)--(P33)--(P34);
\draw (P40)--(P41)--(P42)--(P43)--(P44);

% Nodes
\foreach \X in {00,10,20,30,40,01,11,21,31,41,02,12,22,32,42,03,13,23,33,43,04,14,24,34,44}
  \fill (P\X) circle (0.45pt);

\end{tikzpicture}
\caption{A representative distorted quadrilateral mesh with alternating compression--stretch pattern.}
\label{fig:distortedmesh}
\end{figure}
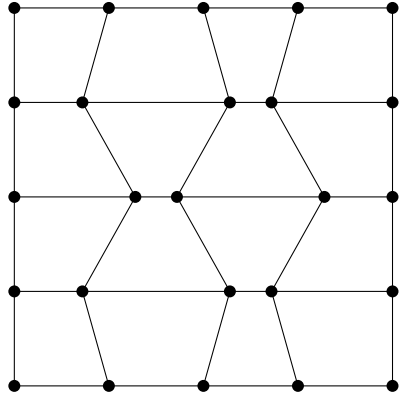

%%%%%%%%%%%%%%%%%%%%%%%%%%%%%%%%%%%%%%%%%%%%%%%%%%%%%%%%%%%%
\subsection{Compatibility defect}

One of the principal motivations for the modified construction is the
loss of the compatibility property
\[
\div V_h \subset W_h
\]
for the classical Piola-mapped $RT_1$ space on distorted quadrilateral
meshes.
To quantify this effect, we consider the defect
\[
\eta_h(\bv_h)
=
\frac{
\|(I-P_{W_h})\div \bv_h\|_{L^2(\Omega)}
}{
\|\div \bv_h\|_{L^2(\Omega)}
},
\]
where $P_{W_h}$ denotes the $L^2$ projection onto the pressure space
$W_h$.

For each basis function of the velocity space, the quantity
$\eta_h(\bv_h)$ measures the component of the divergence lying outside
the discrete pressure space.
The values reported below are averaged over all local basis functions
and mesh elements.

\begin{table}[htbp]
\centering
\caption{Compatibility defect for RT$_1$ and MRT$_1$.}
\label{tab:defect}
\begin{tabular}{c|cc}
\hline
$N$
&
RT$_1$
&
MRT$_1$
\\
\hline
\multicolumn{3}{c}{$\alpha=0.00$}
\\
\hline
4  & $7.94\times10^{-8}$ & $7.81\times10^{-8}$
\\
8  & $7.94\times10^{-8}$ & $7.81\times10^{-8}$
\\
16 & $7.94\times10^{-8}$ & $0$
\\
32 & $7.94\times10^{-8}$ & $5.79\times10^{-8}$
\\
\hline
\multicolumn{3}{c}{$\alpha=0.60$}
\\
\hline
4  & $3.56\times10^{-1}$ & $8.61\times10^{-8}$
\\
8  & $3.56\times10^{-1}$ & $3.88\times10^{-8}$
\\
16 & $3.56\times10^{-1}$ & $0$
\\
32 & $3.56\times10^{-1}$ & $5.79\times10^{-8}$
\\
\hline
\end{tabular}
\end{table}

For affine meshes ($\alpha=0$), both $RT_1$ and $MRT_1$
satisfy the compatibility property up to machine precision.
However, for distorted meshes ($\alpha=0.60$),
the compatibility defect of the classical $RT_1$ space remains
bounded away from zero, while the modified space $MRT_1$
restores the compatibility relation to machine precision.
This experiment directly confirms the theoretical discussion
following \eqref{Falk}.

%%%%%%%%%%%%%%%%%%%%%%%%%%%%%%%%%%%%%%%%%%%%%%%%%%%%%%%%%%%%
\subsection{Interpolation errors}

We next study the approximation behavior of the interpolation operators.
The exact vector field is chosen sufficiently smooth so that the
approximation properties are dominated by the geometry of the mesh.

Table~\ref{tab:divinterp} reports the divergence interpolation errors
for $RT_1$ and $MRT_1$ on distorted meshes.

\begin{table}[htbp]
\centering
\caption{Divergence interpolation errors for $RT_1$ and $MRT_1$.}
\label{tab:divinterp}
\begin{tabular}{c|cc|cc}
\hline
&
\multicolumn{2}{c|}{RT$_1$}
&
\multicolumn{2}{c}{MRT$_1$}
\\
$N$&error&order&error&order
\\
\hline
\multicolumn{5}{c}{$\alpha=0.60$}\\
\hline
4&$7.93949\times10^{-2}$&---&$1.29569\times10^{-2}$
&---\\
8&$3.74578\times10^{-2}$&1.084&$3.23921\times10^{-3}$&2.000
\\
16&$1.81729\times10^{-2}$&1.043&$8.09804\times10^{-4}$&2.000
\\
32&$8.94803\times10^{-3}$&1.022&$2.02451\times10^{-4}$&2.000
\\
\hline
\end{tabular}
\end{table}

The results clearly show the loss of optimal divergence approximation
for the classical $RT_1$ element on distorted quadrilateral meshes.
The observed convergence rate is approximately first order,
which is consistent with the theoretical estimate discussed after
\eqref{Falk} and with the analysis of Arnold, Boffi, and Falk \cite{arnold2002approximation,arnold2005quadrilateral}.
In contrast, the modified space $MRT_1$ restores the expected
second-order convergence in the divergence norm.

Thus the modification successfully compensates for the geometric
distortion introduced by the bilinear mapping while preserving the
degrees of freedom and dimension of the classical $RT_1$ space.

%%%%%%%%%%%%%%%%%%%%%%%%%%%%%%%%%%%%%%%%%%%%%%%%%%%%%%%%%%%%
\subsection{Mixed finite element approximation}

The previous experiments show that the modified space restores the
compatibility relation and recovers optimal divergence interpolation
behavior on distorted meshes.
We next investigate how these geometric effects influence the full
mixed finite element approximation for elliptic problems.

It is important to note that the loss of compatibility in the classical
$RT_1$ space does not necessarily imply catastrophic deterioration of
the PDE solution itself. Depending on the structure of the continuous
solution and the diffusion tensor, the incompatible divergence modes
may be only weakly excited. Consequently, both $RT_1$ and $MRT_1$
may still exhibit similar convergence behavior for certain model
problems, even though their underlying approximation mechanisms differ.

Thus the interpolation experiments provide a clearer numerical
manifestation of the geometric compatibility defect. The classical
$RT_1$ space produces divergence components lying outside the discrete
pressure space on distorted quadrilateral meshes, whereas the modified
space $MRT_1$ restores the compatibility relation
\[
\div V_h \subset W_h
\]
to machine precision. This restoration is reflected directly in the
optimal divergence interpolation behavior observed in the previous
subsection.

Finally, we consider the mixed approximation of the elliptic problem
\[
-\div(K\nabla p)=f
\quad\text{in }\Omega,
\]
with exact solution
\[
p=(x-x^2)(y-y^2),
\]
and anisotropic diffusion tensor
\[
K=
\begin{pmatrix}
10000 & 0\\
0 & 1
\end{pmatrix}.
\]

The corresponding flux is
\[
{\bf u}=-K\nabla p.
\]

Table~\ref{tab:pde} reports the errors for the mixed finite element
approximation on distorted meshes.

\begin{table}[htbp]
\centering
\caption{Divergence errors for the mixed finite element approximation.}
\label{tab:pde}
\begin{tabular}{c|cc|cc}
\hline&\multicolumn{2}{c|}{RT$_1$ divergence}
&\multicolumn{2}{c}{MRT$_1$ divergence}
\\
$N$&error&order&error&order
\\
\hline\multicolumn{5}{c}{$\alpha=0.90$}\\
\hline
4&$9.50664\times10^{1}$&---&$9.31541\times10^{1}$&---
\\
8&$2.38760\times10^{1}$&1.993&$2.32913\times10^{1}$&2.000
\\
16&$5.97635\times10^{0}$&1.998&$5.82303\times10^{0}$&2.000
\\
32&$1.49456\times10^{0}$&2.000&$1.45577\times10^{0}$&2.000
\\
\hline
\end{tabular}
\end{table}

The numerical experiments therefore suggest that the primary
advantage of the modified space is structural rather than merely
empirical. In particular, the modified construction restores the discrete geometric compatibility relation independently of the specific PDE being solved. The space $MRT_1$ restores the discrete compatibility
relation on distorted quadrilateral meshes while preserving the
dimension and degrees of freedom of the classical $RT_1$ space.
The resulting improvement is most clearly visible in the divergence
interpolation behavior, while the impact on the full PDE solution
depends on how strongly the incompatible divergence modes are
excited by the underlying problem.

\section{Conclusion}
We have presented a framework for constructing modified
mixed finite element spaces of Raviart\nobreakdash--Thomas type
on quadrilateral meshes. The modification
introduces geometry-dependent correction terms that compensate
for distortion effects arising from the bilinear mapping between
the reference square and a general quadrilateral element. The modified
spaces preserve the dimension and degrees of freedom of the classical
Raviart--Thomas spaces while restoring desirable approximation
properties under standard shape--regularity assumptions.

The construction was illustrated for the lowest two orders,
leading to modified $RT_0$ and $RT_1$ elements. For the $RT_1$
case, the numerical experiments on distorted quadrilateral meshes
confirmed the predicted second-order interpolation behavior and
demonstrated that the modified formulation restores the compatibility relation and
maintains optimal divergence behavior under increasing geometric distortion.

In addition, the same modification framework can be interpreted
as a local post-processing procedure for solutions obtained with
the classical $RT_0$ method. This provides improved local
approximations without increasing the number of global degrees
of freedom. The approach may also be extended to other finite
element spaces conforming in $\bH(\div)$ on quadrilateral meshes.

The modified Raviart--Thomas spaces constructed in this paper may also
be useful for mixed formulations of the Stokes equations on distorted
quadrilateral meshes. Since the modified spaces restore the compatibility
property between the discrete velocity and pressure spaces while preserving
the classical degrees of freedom, the standard mixed stability framework
is expected to remain applicable. A detailed analysis for Stokes problems,
including discrete inf-sup stability and optimal error estimates, will be
considered in future work.

Another natural direction is the extension of the present geometric
correction framework to three-dimensional hexahedral meshes. In that
setting the Piola transformation involves a trilinear mapping and a
more complicated Jacobian structure. Nevertheless, the present work
suggests that suitable geometry-dependent correction terms may again
restore compatibility between the discrete velocity and pressure spaces.
}

\bibliographystyle{siam}
\bibliography{references}

\end{document}